\documentclass[11pt]{amsart}
\usepackage[english]{babel}
\usepackage[T1]{fontenc}
\usepackage[latin1]{inputenc}
\usepackage{graphicx}
\usepackage{amsmath,amssymb,amsthm,mathrsfs,txfonts,ifthen,xr}
\usepackage{soul}
\usepackage{array, multirow}
\usepackage{stmaryrd}
\usepackage{color}
\usepackage{hyperref}
\usepackage{enumerate}
\usepackage[all]{xy}
\usepackage{tikz-cd}
\usetikzlibrary{decorations.pathreplacing}

\usepackage{ytableau}

\theoremstyle{plain}
\newtheorem{theo}{Theorem}[section]
\newtheorem*{theo*}{Theorem}

\theoremstyle{definition}
\newtheorem{definition}[theo]{Definition}

\theoremstyle{remark}
\newtheorem{remark}[theo]{Remark}

\def\A{{\rm A}}
\def\B{{\rm B}}
\def\C{{\rm C}}
\def\D{{\rm D}}

\def\U{{\rm U}}

\def\X{{\rm X}}
\def\Y{{\rm Y}}

\def\GL{{\rm GL}}	
\def\Id{{\rm Id}}

\def\Mat{{\rm Mat}}

\def\log{{\rm log}}

\def\Spec{{\rm Spec}}
\def\diag{{\rm diag}}

\allowdisplaybreaks
\title{On The Linearization of Alternative Means}
\author{Raluca Dumitru}
\address{Department of Mathematics and Statistics\\ University of North Florida \\ 1 UNF Drive \\ Jacksonville \\ FL 32224 \\ USA}
\email{raluca.dumitru@unf.edu}
\author{Jose Franco}
\address{Department of Mathematics and Statistics\\ University of North Florida \\ 1 UNF Drive \\ Jacksonville \\ FL 32224 \\ USA}
\email{jose.franco@unf.edu}
\author{Allan Merino}
\address{Department of Mathematics and Statistics\\ University of North Florida \\ 1 UNF Drive \\ Jacksonville \\ FL 32224 \\ USA}
\email{allan.merino@unf.edu}

\keywords{Alternative Means, Wasserstein Mean, Linearization}

\subjclass[2010]{Primary: 15A42; Secondary: 47A63.}

\date{}

\begin{document}

\maketitle

\begin{abstract}

Alternative means have recently attracted considerable attention in matrix analysis and operator theory. In this paper, we investigate the linearization problem for alternative means, namely the question of determining when a mean can be expressed as an affine combination of the matrices under consideration. We first prove a conjecture of Choi, Kim, and Lim for the Wasserstein mean. More precisely, we show that the Wasserstein mean $\A \diamond \B$ is linearizable if and only if $\A\B = \B\A$ and $\left|\Spec(\A^{-1}\B)\right| \leq 2$. We further establish a general rigidity theorem for a large class of alternative means. Specifically, we prove that an analogous characterization holds whenever the representing function $f$ is of the form $f(x) = \sqrt{h(x)}$, where $h$ is a non-affine operator monotone function. As consequences, we obtain linearization criteria for several families of alternative means, including logarithmic, harmonic, and power-type means.

\end{abstract}

\tableofcontents

\section{Introduction}

The theory of operator means introduced by Kubo and Ando \cite{KUBOANDO} provides a powerful framework for studying binary operations on the cone of positive definite matrices. A fundamental result of their theory states that every operator means is uniquely determined by an operator monotone function through a functional calculus formula. Classical examples include the arithmetic, geometric, harmonic, logarithmic, and power means, which play an important role in matrix analysis, operator theory, differential geometry, and applications\,.

\noindent A natural problem in the theory of matrix means is to determine when a given mean can be expressed in a simpler form involving only the matrices under consideration. Motivated by this question, the authors studied in \cite{RJA} the linearization problem for Kubo-Ando means. More precisely, given a Kubo-Ando mean $\sigma$, one asks for which pairs $\left(\A\,, \B\right) \in \mathscr{P}_{n}(\mathbb{D}) \times \mathscr{P}_{n}(\mathbb{D})$, there exist $a\,, b \in \mathbb{R}$ such that
\begin{equation*}
\A \sigma \B = a\A + b\B\,.
\end{equation*}
It was shown in \cite{RJA} that if the representing function of $\sigma$ is non-affine, then $\A \sigma \B$ is linearizable if and only if $\left|\Spec(\A^{-1}\B)\right| \leq 2$. Thus, for Kubo-Ando means, the linearization problem is completely determined by the spectrum of $\A^{-1}\B$\,.

\noindent More recently, the notion of an alternative mean was introduced in \cite{DFKC}. Given a continuous function $f: \left(0\,, \infty\right) \to \left(0\,, \infty\right)$ satisfying $f(1)=1$, the associated alternative mean is defined by
\begin{equation*}
\A \hat{\sigma}_{f} \B =  f\left(\A^{-1} \sharp \B\right)\A f\left(\A^{-1} \sharp \B\right)\,,
\end{equation*}
where $\sharp$ denotes the geometric mean. Alternative means may be viewed as a natural extension of the Kubo-Ando framework, obtained by replacing the classical functional calculus variable $\A^{-\frac{1}{2}}\B\A^{-\frac{1}{2}}$ with the geometric quantity $\A^{-1}\sharp \B$. In the commutative case, alternative means satisfies
\begin{equation*}
\A\sigma_f \B = \left(\A^{\frac{1}{2}}\hat{\sigma}_{f} \B^{\frac{1}{2}}\right)^{2}\,,
\end{equation*}
where $\sigma_{f}$ is the Kubo-Ando mean with representing function $f$, and where in the case of the geometric mean, both the alternative mean and Kubo-Ando means match\,.

\noindent A particularly important example is the Wasserstein mean, corresponding to the function
\begin{equation*}
f(t)=\frac{1+t}{2}\,.
\end{equation*}
In \cite{CHOIKIMLIM}, the authors proved that if $\A\B=\B\A$ and $\left|\Spec(\A^{-1}\B)\right|\leq 2$, then the Wasserstein mean is linearizable, and conjectured that the converse also holds. The first goal of this paper is to prove this conjecture\,.

\begin{theo*}
Let $\A\,, \B \in \mathscr{P}_{n}(\mathbb{D})$. Then the Wasserstein mean $\A \diamond \B$ is linearizable if and only if $\A\B=\B\A$ and $\left|\Spec(\A^{-1}\B)\right|\leq 2$\,.

\end{theo*}

\noindent The proof of the previous theorem reveals a more general rigidity phenomenon. Indeed, we show that the same conclusion remains valid for a large family of alternative means associated with operator monotone functions. First of all, in \cite{RJA}, we proved that for any alternative mean $\hat{\sigma}_{f}$, if $\A$ and $\B$ are such that $\A\B = \B\A$ and $\left|\Spec(\A^{-1}\B)\right| \leq 2$, then $\A \hat{\sigma}_{f} \B$ is linearizable\,.

\noindent In the paper, we proved that the converses also hold if $h: \left(0\,, \infty\right)\to \left(0\,, \infty\right)$ is a non-affine operator monotone function and $f(x)=\sqrt{h(x)}$\,.

\begin{theo*}

Let $h: \left(0\,, \infty\right) \to \left(0\,, \infty\right)$ be a non-affine operator monotone function and let $f(x) = \sqrt{h(x)}$. If the alternative mean $\A \hat{\sigma}_{f} \B$ is linearizable, then $\A\B=\B\A$ and $\left|\Spec(\A^{-1}\sharp\B)\right|\leq 2$\,.

\end{theo*}

\noindent As consequences, we obtain linearization criteria for several families of alternative means, including harmonic, logarithmic, and power-type means. These results considerably extend the conjecture of Choi, Kim, and Lim and show that the rigidity of the Wasserstein mean is in fact part of a broader phenomenon. The linearization problems have recently been investigated for other classes of matrix means (see \cite{KIM2026}). In particular, Gan, Kim, and Mer studied the weighted spectral geometric mean and obtained several characterizations of its linearity\,.

\noindent Note that the assumptions of the previous theorem cannot be completely removed. Indeed, if $f(t)=t$, then it follows from Ricatti's equation that
\begin{equation*}
\A \hat{\sigma}_{f} \B = \left(\A^{-1}\sharp\B\right)\A\left(\A^{-1}\sharp\B\right) = \B\,,
\end{equation*}
so the corresponding alternative mean is always linearizable, independently of the matrices $\A$ and $\B$. This naturally leads to the following question\,.

\medskip

\noindent \textbf{\underline{Question :}} Can we characterize all functions $f: \left(0\,, \infty\right) \to \left(0\,, \infty\right)$ such that $\A \hat{\sigma}_{f} \B$ is linearizable if and only if $\A\B=\B\A$ and $\left|\Spec(\A^{-1}\sharp\B)\right|\leq 2$\,?

\medskip

\noindent The paper is organized as follows. In Section 2, we recall the necessary background on Kubo-Ando means and alternative means. Section 3 is devoted to the proof of the Choi-Kim-Lim conjecture for the Wasserstein mean. In Section 4, we establish our general rigidity theorem and discuss several examples and applications\,.

\section{Preliminaries}

Let $\mathbb{D} \in \left\{\mathbb{R}\,, \mathbb{C}\,, \mathbb{H}\right\}$. We denote by $\Mat_{n}(\mathbb{D})$ the set of $n$ by $n$ matrices with entries in $\mathbb{D}$, and by $\GL_{n}(\mathbb{D})$ the set of invertible matrices in $\Mat_{n}(\mathbb{D})$. For a matrix $\A \in \Mat_{n}(\mathbb{D})$, we denote by $\A^{*}$ the matrix in $\Mat_{n}(\mathbb{D})$ given by
\begin{equation*}
\A^{*} = \overline{\A}^{t}\,.
\end{equation*}
Let $\U_{\mathbb{D}}$ and $\mathfrak{p}_{n}(\mathbb{D})$ the subsets of $\GL_{n}(\mathbb{D})$ and $\Mat_{n}(\mathbb{D})$ respectively given by
\begin{equation*}
\U_{\mathbb{D}} = \left\{g \in \GL_{n}(\mathbb{D})\,, gg^{*} = \Id_{n}\right\}\,, \qquad \mathfrak{p}_{n}(\mathbb{D}) := \left\{\X \in \Mat_{n}(\mathbb{D})\,, \X = \X^{*}\right\}\,.
\end{equation*}
We denote by $\mathscr{P}_{n}(\mathbb{D})$ the subset of $\mathfrak{p}_{n}(\mathbb{D})$ consisting of positive matrices, i.e.
\begin{equation*}
\mathscr{P}_{n}(\mathbb{D}) := \left\{\X \in \mathfrak{p}_{n}(\mathbb{D})\,, \X > 0\right\}\,,
\end{equation*}
and by $\preceq$ the Loewner order on $\mathscr{P}_{n}(\mathbb{D})$ (see \cite{BHATIA}).

\begin{remark}

\begin{enumerate}
\item It is well-known (see \cite{BHATIA} and \cite{ZHANG} for the case $\mathbb{D} = \mathbb{H}$) that the exponential map
\begin{equation*}
\exp: \mathfrak{p}_{n}(\mathbb{D}) \to \mathscr{P}_{n}(\mathbb{D})
\end{equation*}
is bijective and we denote by $\log: \mathscr{P}_{n}(\mathbb{D}) \to \mathfrak{p}_{n}(\mathbb{D})$ its inverse. For all $\X \in \mathscr{P}_{n}(\mathbb{D})$ and $t \in \mathbb{R}$, we denote by $\X^{t}$ the matrix in $\mathscr{P}_{n}(\mathbb{D})$ given by
\begin{equation*}
\X^{t} = \exp\left(t\log(\X)\right)\,.
\end{equation*}
\item (Spectral decomposition) Every matrix $\X \in \mathscr{P}_{n}(\mathbb{D})$ can be written as
\begin{equation*}
\X = \U\D\U^{*}\,,
\end{equation*}
with $\U \in \U_{\mathbb{D}}$ and $\D = \diag\left(d_{i}\right)$ a diagonal matrix is $\Mat_{n}(\mathbb{R})$, with $d_{i} > 0$. The positive scalar $d_{i}$ are the eigenvalues of $\X$, and in the case $\mathbb{D} = \mathbb{H}$, we mean right eigenvalues. For a matrix $\X \in \mathscr{P}_{n}(\mathbb{D})$, we denote by $\Spec(\X)$ the spectrum of $\X$, i.e. the set of eigenvalues of $\X$\,.
\item Note that for all $\X\,, \Y \in \mathscr{P}_{n}(\mathbb{D})$, we have $\Spec(\X\Y) = \Spec(\Y\X)$. Moreover, for all $\U \in \GL_{n}(\mathbb{D})$ and $\X \in \mathscr{P}_{n}(\mathbb{D})$, we get $\U\X\U^{*} \in \mathscr{P}_{n}(\mathbb{D})$. In particular, for all $\A\,, \B \in \mathscr{P}_{n}(\mathbb{D})$, the matrix $\A^{-\frac{1}{2}}\B\A^{-\frac{1}{2}} \in \mathscr{P}_{n}(\mathbb{D})$ and $\Spec(\A^{-\frac{1}{2}}\B\A^{-\frac{1}{2}}) = \Spec(\A^{-1}\B)$\,.
\end{enumerate}

\label{RemarkSpectralDecomposition}

\end{remark}

\noindent In \cite{KUBOANDO}, Kubo and Ando introduced an axiomatic theory of means on the cone $\mathscr{P}_{n}(\mathbb{D})$. More precisely, a Kubo-Ando mean is a continuous map 
\begin{equation*}
\sigma: \mathscr{P}_{n}(\mathbb{D}) \times \mathscr{P}_{n}(\mathbb{D}) \to \mathscr{P}_{n}(\mathbb{D})
\end{equation*}
such that 
\begin{enumerate}
\item If $\A \preceq \B$ and $\C \preceq \D$, then $\A \sigma \C \preceq \B \sigma \D$\,,
\item For all $\A\,, \B \in \mathscr{P}_{n}(\mathbb{D})$ and $\X \in \GL_{n}(\mathbb{D})$, we get 
\begin{equation*}
\X\left(\A\sigma\B\right)\X^{*} = \left(\X\A\X^{*}\right)\sigma\left(\X\B\X^{*}\right)\,,
\end{equation*}
\item $\Id_{n} \sigma \Id_{n} = \Id_{n}$\,.
\end{enumerate}
A fundamental result of their theory states that every mean is uniquely determined by an operator monotone function $f: \left(0\,, \infty\right) \to \left(0\,, \infty\right)$ through a functional calculus formula. More precisely, for a Kubo-Ando mean $\sigma$ on $\mathscr{P}_{n}(\mathbb{D})$, there exists an operator monotone function $f:\left(0\,, \infty\right) \to \left(0\,, \infty\right)$ with $f(1) = 1$ and such that
\begin{equation*}
\A \sigma \B = \A^{\frac{1}{2}}f\left(\A^{-\frac{1}{2}}\B\A^{-\frac{1}{2}}\right)\A^{\frac{1}{2}}\,, \qquad \left(\A\,, \B \in \mathscr{P}_{n}(\mathbb{D})\right)\,.
\end{equation*}

\begin{remark}

\begin{enumerate}
\item Let $\sigma$ be a Kubo-Ando mean on $\mathscr{P}_{n}(\mathbb{D})$. For two matrices $\A\,, \B \in \mathscr{P}_{n}(\mathbb{D})$, we say that $\A \sigma \B$ is linearizable if there exists $a\,, b \in \mathbb{R}$ such that
\begin{equation*}
\A \sigma \B = a\A + b\B\,.
\end{equation*}
In a recent paper (see \cite{RJA}), we proved that if the representing function $f$ is not affine, then $\A \sigma \B$ is linearizable if and only if $\left|\Spec(\A^{-1}\B)\right| \leq 2$. In particular, the linearization problem is well-understood for Kubo-Ando means.
\item We denote by $\sharp$ the geometric mean, i.e. the Kubo-ando mean corresponding to $f(x) = \sqrt{x}$. In particular, 
\begin{equation*}
\A \sharp \B = \A^{\frac{1}{2}}\left(\A^{-\frac{1}{2}}\B\A^{-\frac{1}{2}}\right)^{\frac{1}{2}}\A^{\frac{1}{2}}\,, \qquad \left(\A\,, \B \in \mathscr{P}_{n}(\mathbb{D})\right)\,.
\end{equation*}
The geometric has a nice characterization. Indeed, $\A \sharp \B$ is the unique solution of the Ricatti equation
\begin{equation*}
\X\A^{-1}\X = \B\,.
\end{equation*}
In particular, for all $\A\,, \B \in \mathscr{P}_{n}(\mathbb{D})$, we have
\begin{equation}
\left(\A^{-1} \sharp \B\right)\A\left(\A^{-1} \sharp \B\right) = \B\,.
\label{EquationRicatti}
\end{equation}
The equation \eqref{EquationRicatti} will play an important role in the next sections\,.
\end{enumerate}

\label{RemarkRicatti}

\end{remark}

\noindent In this paper, we study the linearization problem for alternative means. 

\begin{definition}

Let $f: \left(0\,, \infty\right) \rightarrow \left(0\,, \infty\right)$ be a continuous function such that $f(1)=1$. The alternative mean associated with $f$ is the binary operation
\begin{equation*}
\hat{\sigma}_{f}: \mathscr{P}_{n}(\mathbb{D}) \times \mathscr{P}_{n}(\mathbb{D}) \to \mathscr{P}_{n}(\mathbb{D})
\end{equation*}
defined by
\begin{equation*}
\A \hat{\sigma}_{f} \B = f\left(\A^{-1} \sharp \B\right)\A f\left(\A^{-1}\sharp \B\right)\,, \qquad \left(\A\,, \B \in \mathscr{P}_{n}(\mathbb{D})\right)\,.
\end{equation*}

\end{definition}

\noindent We now recall a result of \cite{RJA}\,.

\begin{theo}

Let $f: \left(0\,, \infty\right) \rightarrow \left(0\,, \infty\right)$ be a continuous function such that $f(1)=1$. Then for all $\A\,, \B \in \mathscr{P}_{n}(\mathbb{D})$ such that $\A\B = \B\A$ and $\left|\Spec(\A^{-1}\B)\right| \leq 2$, $\A \hat{\sigma}_{f} \B$ is linearizable\,.

\label{TheoremRFA1}

\end{theo}

\noindent The proof of Theorem \ref{TheoremRFA1} can be found in \cite{CHOIKIMLIM} for the Wasserstein mean (i.e. for $f(t) = \frac{1+t}{2}$). In \cite{CHOIKIMLIM}, the authors conjectured that for the Wasserstein mean, the implication given in Theorem \ref{TheoremRFA1} is in fact an equivalence . In this paper, we prove this conjecture and show that such a statement remains valid for a large family of alternative means\,.

\begin{remark}

\begin{enumerate}
\item In \cite{RJA}, we obtained an explicit decomposition of any alternative mean. Indeed, let $f: \left(0\,, \infty\right) \rightarrow \left(0\,, \infty\right)$ be a continuous function such that $f(1)=1$ and let $\A\,, \B \in \mathscr{P}_{n}(\mathbb{D})$ such that $\left|\Spec(\A^{-1} \sharp \B)\right| = r$, then 
\begin{equation*}
\A \hat{\sigma}_{f} \B = \sum\limits_{i = 0}^{r-1}\sum\limits_{j = 0}^{r-1} d_{i}d_{j} \left(\A^{-1} \sharp \B\right)^{i}\A\left(\A^{-1} \sharp \B\right)^{j}\,,
\end{equation*}
where the coefficients $d_{i}$ only depends on $f$ and the eigenvalues of $\A^{-1} \sharp \B$. However, we don't have, in general, an easy description of $\left(\A^{-1} \sharp \B\right)^{i}$\,.
\item The commutative case is way easier to describe. Indeed, if $\A\B = \B\A$, then 
\begin{equation*}
\A^{-1} \sharp \B = \left(\A^{-1}\B\right)^{\frac{1}{2}}\,.
\end{equation*}
Moreover, it follows from functional calculus that
\begin{equation*}
\A \hat{\sigma}_{f} \B = \A f\left(\left(\A^{-1}\B\right)^{\frac{1}{2}}\right)^{2}\,,
\end{equation*}
so if $f$ if operator monotone, then $\A \hat{\sigma}_{f} \B$ is $\A \sigma_{g} \B$, where $\sigma_{g}$ is the Kubo-Ando mean with representing function $g(t) = f(\sqrt{t})^{2}$\,.
\end{enumerate}

\end{remark}

\section{Linearization of the Wasserstein mean}

\label{SectionThree}

In this section, we take $f(t) = \frac{1+t}{2}$. We denote by $\diamond$ the corresponding alternative mean, i.e.
\begin{equation*}
\A \diamond \B = \frac{1}{4}\left(\Id_{n} + \A^{-1} \sharp \B\right)\A\left(\Id_{n} + \A^{-1} \sharp \B\right)\,.
\end{equation*}

\noindent The following theorem was conjectured in \cite{CHOIKIMLIM}\,.

\begin{theo}

Let $\A\,, \B \in \mathscr{P}_{n}(\mathbb{D})$. Assume that $\A \diamond \B$ is linearizable. Then $\left|\Spec(\A^{-1} \sharp \B)\right| \leq 2$ and $\A\B = \B\A$\,.

\label{TheoremConjectureCKL}

\end{theo}

\begin{proof}

Let $\Y = \A^{-1} \sharp \B$. Using that $f(t) = \frac{1+t}{2}$, we get
\begin{equation*}
\A \diamond \B = \frac{1}{4}\left(\Id_{n} + \Y\right)\A\left(\Id_{n} + \Y\right) = \frac{1}{4}\left(\A + \A\Y + \Y\A + \Y\A\Y\right)\,.
\end{equation*}
By assumption, there exists $a\,, b \in \mathbb{R}$ such that 
\begin{equation*}
\A \diamond \B = a\A + b\B\,.
\end{equation*}
Moreover, using that $\Y$ is the unique solution of $\Y\A\Y = \B$ (see Equation \eqref{EquationRicatti}), it follows that 
\begin{equation*}
4a\A + 4b\Y\A\Y = \A + \A\Y + \Y\A + \Y\A\Y\,,
\end{equation*}
i.e.
\begin{equation*}
\alpha\A + \beta\Y\A\Y = \A\Y + \Y\A\,.
\end{equation*}
with $\alpha = 4a-1$ and $\beta = 4b-1$. Using that $\Y \in \mathscr{P}_{n}(\mathbb{D})$, it follows from Remark \ref{RemarkSpectralDecomposition} that there exists $\U \in \U_{\mathbb{D}}$ and $\D = \diag\left(d_{1}\,, \ldots\,, d_{n}\right) \in \GL_{n}(\mathbb{R})$ such that $\Y = \U\D\U^{*}$. Let $\A' = \U^{*}\A\U \in \mathscr{P}_{n}(\mathbb{D})$. We have
\begin{equation*}
\alpha\A + \beta\Y\A\Y = \alpha\U\A'\U^{*} + \beta\U\D\U^{*}\U\A'\U^{*}\U\D\U^{*} = \U\left(\alpha\A' + \beta \D\A'\D\right)\U^{*}
\end{equation*}
and
\begin{equation*}
\A\Y + \Y\A = \U\A'\U^{*}\U\D\U^{*} + \U\D\U^{*}\U\A'\U^{*} = \U\left(\A'\D + \D\A'\right)\U^{*}\,,
\end{equation*}
i.e.
\begin{equation*}
\alpha\A' + \beta \D\A'\D = \A'\D + \D\A'\,.
\end{equation*}
For all $1 \leq i\,, j \leq n$, we have
\begin{equation*}
\left(\alpha\A' + \beta \D\A'\D\right)_{i,j} = \alpha a'_{i,j} + \beta d_{i}d_{j}a'_{i,j} = \left(\alpha + \beta d_{i}d_{j}\right)a'_{i,j}
\end{equation*}
and 
\begin{equation*}
\left(\A'\D + \D\A'\right)_{i,j} = a'_{i,j}d_{i} + d_{j}a'_{i,j} = \left(d_{i}+d_{j}\right)a'_{i,j}\,,
\end{equation*}
i.e.
\begin{equation}
\left(d_{i}+d_{j}-\alpha-\beta d_{i}d_{j}\right)a'_{i,j} = 0\,, \qquad \left(1 \leq i\,, j \leq n\right)\,.
\label{EquationSpectrum}
\end{equation}
Using that $\A' > 0$, it follows that for all $1 \leq i \leq n$, $e^{*}_{i}\A'e_{i} = a'_{i,i} > 0$. In particular, we get from Equation \eqref{EquationSpectrum} that
\begin{equation*}
2d_{i} - \alpha - \beta d^{2}_{i} = 0\,, \qquad \left(1 \leq i \leq n\right)\,.
\end{equation*}
Therefore, every eigenvalue of $\Y$ is a solution of $\beta t^{2} - 2t + \alpha = 0$, i.e.
\begin{equation*}
\left|\Spec(\A^{-1} \sharp \B)\right| \leq 2\,.
\end{equation*}
We now prove the commutativity. We distinguish two cases\,.
\begin{enumerate}
\item If $\left|\Spec(\A^{-1} \sharp \B)\right| = 1$, then $\A^{-1} \sharp \B \in \mathbb{R}^{*}_{+}\Id_{n}$, and therefore $\B$ is a multiple of $\A$, i.e. $\A\B = \B\A$\,.
\item If $\left|\Spec(\A^{-1} \sharp \B)\right| = 2$, then $\beta \neq 0$. In particular, we can assume that $\D = \diag\left(d_{1}\Id_{n_{1}}\,, d_{2}\Id_{n_{2}}\right)$, with $n_{1}\,, n_{2} \in \mathbb{N}$ such that $n_{1} + n_{2} = n$. The goal is to prove that for all $1 \leq i \leq n_{1}$ and $n_{1}+1 \leq j \leq n$, $a'_{i,j} = a'_{j,i} = 0$. Suppose that there exists $1 \leq i \leq n_{1}$ and $n_{1}+1 \leq j \leq n$ such that $a'_{i,j} \neq 0$ (the case $a'_{j,i} \neq 0$ is similar). Therefore, it follows from Equation \eqref{EquationSpectrum} that
\begin{equation}
d_{1}+d_{2}-\alpha-\beta d_{1}d_{2} = 0\,.
\label{EquationD1D2}
\end{equation}
By definition, $d_{1}$ and  $d_{2}$ are the two solutions of 
\begin{equation*}
\beta t^{2} - 2t + \alpha = 0\,,
\end{equation*}
or equivalently, using that $\beta \neq 0$,
\begin{equation}
t^{2} - \frac{2}{\beta}t + \frac{\alpha}{\beta} = 0\,.
\label{EquationDegree2}
\end{equation}
In particular, we get
\begin{equation}
d_{1} + d_{2} = \frac{2}{\beta}\,, \qquad \frac{\alpha}{\beta} = d_{1}d_{2}\,.
\label{EquationD1D2V2}
\end{equation}
Thus, it follows from Equations \eqref{EquationD1D2} and  \eqref{EquationD1D2V2} that
\begin{equation*}
\frac{2}{\beta} - \alpha - \beta\frac{\alpha}{\beta} = 0\,,
\end{equation*}
i.e.
\begin{equation*}
\alpha\beta = 1\,.
\end{equation*}
The discriminant $\Delta$ of the quadratic polynomial given in \eqref{EquationDegree2} is
\begin{equation*}
\Delta = \frac{4}{\beta^{2}} - 4\frac{\alpha}{\beta} = \frac{4}{\beta^{2}}\left(1 - \alpha\beta\right)\,,
\end{equation*} 
so $\Delta = 0$, contradicting $\left|\Spec(\A^{-1} \sharp \B)\right| = 2$. Therefore, it follows that the matrix $\A'$ is of the form 
\begin{equation*}
\A' = \begin{pmatrix} \Omega_{1} & 0 \\ 0 & \Omega_{2} \end{pmatrix}\,, \qquad \left(\Omega_{1} \in \Mat_{n_{1}}(\mathbb{D})\,, \Omega_{2} \in \Mat_{n_{2}}(\mathbb{D})\right)\,.
\end{equation*}
so $\D\A' = \A'\D$. Finally, using that
\begin{equation*}
\D\A' = \D\U^{*}\A\U = \U^{*}\left(\U\D\U^{*}\right)\A\U = \U^{*}\Y\A\U\,,
\end{equation*}
and 
\begin{equation*}
\A'\D = \U^{*}\A\U\D = \U^{*}\A\left(\U\D\U^{*}\right)\U = \U^{*}\A\Y\U\,,
\end{equation*}
we get that $\A\Y = \Y\A$, and then
\begin{equation*}
\A\B = \A\Y\A\Y = \left(\A\Y\right)\left(\A\Y\right) = \left(\Y\A\right)\left(\Y\A\right) = \left(\Y\A\Y\right)\A = \B\A\,,
\end{equation*}
i.e. $\A$ and $\B$ commute.
\end{enumerate}

\end{proof}

\noindent The following theorem follows from Theorems \ref{TheoremRFA1} and \ref{TheoremConjectureCKL}\,.

\begin{theo}

Let $\A\,, \B \in \mathscr{P}_{n}(\mathbb{D})$. Then $\A \diamond \B$ is linearizable if and only if $\A\B = \B\A$ and $\left|\Spec(\A^{-1} \sharp \B)\right| \leq 2$\,.

\end{theo}

\noindent The goal of the next section is to extend the results of this section to a large family of alternative means\,.

\section{A general result on the linearization of $\A \hat{\sigma}_{f} \B$}

Let $f: \left(0\,, \infty\right) \to \left(0\,, \infty\right)$ be an operator monotone function, and let $\hat{\sigma}_{f}$ be the corresponding alternative mean, i.e.
\begin{equation*}
\A \hat{\sigma}_{f} \B = f(\A^{-1} \sharp \B) \A f(\A^{-1} \sharp \B)\,.
\end{equation*}
We keep the notations of Section \ref{SectionThree}. Let $\Y = \A^{-1} \sharp \B$. Again, using Remark \ref{RemarkSpectralDecomposition}, we write $\Y$ as $\Y = \U\D\U^{*}$, with $\U \in \U_{\mathbb{D}}$, and $\D = \diag\left(d_{1}\,, \ldots\,, d_{n}\right)$, with $d_{i} > 0$, and let $\A' = \U^{*}\A\U$. By functional calculus, we have $f(\Y) = \U f(\D)\U^{*}$, i.e.
\begin{equation}
f(\Y)\A f(\Y) = \U f(\D)\U^{*}\A\U f(\D)\U^{*} = \U\left(f(\D)\A'f(\D)\right)\U^{*}\,.
\label{ExtensionOne}
\end{equation}
Assume that $\A \hat{\sigma}_{f} \B$ is linearizable, i.e. $\A \hat{\sigma}_{f} \B = a\A + b\B$. Using that $\B = \Y\A\Y$ (see Equation \eqref{EquationRicatti}), we get that 
\begin{equation}
a\A + b\B = a\U\A'\U^{*} + b\U\D\U^{*}\A\U\D\U^{*} = \U\left(a\A' + b\D\A'\D\right)\U^{*}\,.
\label{ExtensionTwo}
\end{equation}
In particular, it follows from Equations \eqref{ExtensionOne} and \eqref{ExtensionTwo} that 
\begin{equation*}
f(\D)\A'f(\D) = a\A' + b\D\A'\D\,.
\end{equation*}
For all $1 \leq i\,, j \leq n$, we have 
\begin{equation*}
\left(f(\D)\A'f(\D)\right)_{i,j} = \left(a\A' + b\D\A'\D\right)_{i,j}
\end{equation*}
i.e.
\begin{equation}
a'_{i,j}\left(f(d_{i})f(d_{j}) - a - bd_{i}d_{j}\right) = 0\,.
\label{KeyEquationForAPrime}
\end{equation}
Moreover, using that $\A' \in \mathscr{P}_{n}(\mathbb{D})$, we get $a'_{i,i} > 0$. i.e.
\begin{equation}
f(d_{i})^{2} - a - bd^{2}_{i} = 0\,, \qquad \left(1 \leq i \leq n\right)\,.
\label{EigenvaluesFD}
\end{equation}
In other words, every $\lambda \in \Spec(\Y)$ satisfies
\begin{equation*}
f(\lambda)^{2} = a + b\lambda^{2}\,.
\end{equation*}

\begin{remark}

\begin{enumerate}
\item Let $h: \left(0\,, \infty\right) \to \left(0\,, \infty\right)$ be an operator monotone function. Then there exists $\delta\,, \gamma \in \mathbb{R}_{+}$ such that 
\begin{equation*}
h(x) = \delta + \gamma x + \displaystyle\int_{0}^{\infty} \frac{\lambda x}{x+\lambda}d\mu(\lambda)\,,
\end{equation*}
where $\mu$ is a positive measure on $\left(0\,, \infty\right)$ (see \cite[Theorem~4.41]{HiaiPetz2014}). In particular, $h$ is affine if and only if $\mu = 0$, and if $\mu$ is not $0$, we have 
\begin{equation*}
h'(x) = \gamma + \displaystyle\int_{0}^{\infty} \frac{\lambda^{2}}{(x+\lambda)^{2}} d\mu(\lambda) > 0
\end{equation*}
and
\begin{equation*}
h''(x) = -\displaystyle\int_{0}^{\infty}\frac{\lambda^{2}}{(x+\lambda)^{3}}d\mu(\lambda) < 0\,,
\end{equation*}
i.e. $h$ is strictly increasing and strictly concave.

\noindent Now, let $f: \left(0\,, \infty\right) \to \left(0\,, \infty\right)$ be the function given by
\begin{equation*}
f(x) = \sqrt{h(x)}\,, \qquad \left(x \in \left(0\,, \infty\right)\right)\,.
\end{equation*}
Using that the map $x \to \sqrt{x}$ is operator monotone, it follows that $f$ is operator monotone. Moreover, if $h$ is not affine, then $h$ is strictly concave, i.e. $f^{2}$ is strictly concave\,.
\item Now let $f: \left(0\,, \infty\right) \to \left(0\,, \infty\right)$ be an operator monotone function. Then the transpose $f^{\sharp}$ of $f$ given by
\begin{equation*}
f^{\sharp}(x) = xf(x^{-1})
\end{equation*}
is operator monotone. Therefore
\begin{equation*}
x \to \frac{f(x)}{x}\,, \qquad \left(x \in \left(0\,, +\infty\right)\right)\,,
\end{equation*}
is operator monotone decreasing\,. 
\end{enumerate}

\label{RemarkOperatorMonotone}

\end{remark}

\begin{theo}

Let $h: \left(0\,, \infty\right) \to \left(0\,, \infty\right)$ be a non-affine  operator monotone function such that $h(1) = 1$, and let $f(x) = \sqrt{h(x)}$ as in Remark \ref{RemarkOperatorMonotone}. If $\A \hat{\sigma}_{f} \B$ is linearizable, then $\left|\Spec(\A^{-1} \sharp \B)\right| \leq 2$ and $\A\B = \B\A$\,.

\label{TheoremSquareRootH}

\end{theo}

\begin{proof}

By assumption, there exists $a\,, b \in \mathbb{R}$ such that
\begin{equation*}
\A \hat{\sigma}_{f} \B = a\A + b\B\,.
\end{equation*}

\noindent We first prove that $\left|\Spec(\A^{-1} \sharp \B)\right| \leq 2$. As explained in Equation \eqref{EigenvaluesFD}, any eigenvalue $\lambda \in \Spec(\A^{-1} \sharp \B)$ is a solution of the equation
\begin{equation*}
f(\lambda)^{2} - a - b\lambda^{2} = 0\,.
\end{equation*}
We distinguish three cases\,.
\begin{itemize}
\item If $b < 0$, then $x \to a + bx^{2}$ is decreasing on $\left(0\,, \infty\right)$. The function $f^{2}$ is increasing, so it follows that the equation 
\begin{equation*}
f(\lambda)^{2} = a + b\lambda^{2}
\end{equation*}
has at most one solution, i.e. $\left|\Spec(\A^{-1} \sharp \B)\right| = 1 \leq 2$\,.
\item If $b > 0$, the function $x \to a + bx^{2}$ is convex. Using that $f^{2}$ is strictly concave, it follows that the equation 
\begin{equation*}
f(\lambda)^{2} = a + b\lambda^{2}
\end{equation*}
has at most two solutions, i.e. $\left|\Spec(\A^{-1} \sharp \B)\right| \leq 2$\,.
\item If $b = 0$, then $\A \hat{\sigma}_{f} \B = a\A$, i.e. $a > 0$ and $f(\lambda)^{2} = a$ has a unique solution\,.
\end{itemize}
Finally $\left|\Spec(\A^{-1} \sharp \B)\right| \leq 2$\,.

\medskip

\noindent We now prove that $\A\B = \B\A$. We distinguish two cases\,.
\begin{enumerate}
\item If $\left|\Spec(\A^{-1} \sharp \B)\right| = 1$, then $\B$ is a multiple of $\A$, i.e. $\A$ and $\B$ commute\,.
 \item If $\left|\Spec(\A^{-1} \sharp \B)\right| = 2$, then $\Y$ is of the form $\Y = \U\diag\left(d_{1}\Id_{n_{1}}\,, d_{2}\Id_{n_{2}}\right)\U^{*}$, with $\U \in \U_{\mathbb{D}}$ and $n_{1} + n_{2} = n$. In particular, the eigenvalues $d_{1}$ and $d_{2}$ are such that
\begin{equation}
f(d_{1})^{2} = a + bd^{2}_{1}\,, \qquad f(d_{2})^{2} = a + bd^{2}_{2}\,.
\label{EquationProofOne}
\end{equation}
Let $1 \leq i \leq n_{1}$ and $n_{1} + 1 \leq j \leq n$. Assume that $a'_{i, j} \neq 0$ (the proof for $a'_{j,i} \neq 0$ is similar). In particular, it follows from Equation \eqref{KeyEquationForAPrime} that
\begin{equation}
f(d_{1})f(d_{2}) = a + bd_{1}d_{2}\,.
\label{EquationProofTwo}
\end{equation}
One can see that Equations \eqref{EquationProofOne} and \eqref{EquationProofTwo} implies that $a = 0$. Indeed, it follows from Equation \eqref{EquationProofOne} that
\begin{equation*}
f(d_{1})^{2} f(d_{2})^{2} = \left(a + bd^{2}_{1}\right)\left(a + bd_{2}^{2}\right) = a^{2} + ab\left(d^{2}_{1} + d^{2}_{2}\right) + b^{2}d^{2}_{1}d^{2}_{2}\,,
\end{equation*}
and squaring Equation \eqref{EquationProofTwo} gives
\begin{equation*}
\left(f(d_{1})f(d_{2})\right)^{2} = \left(a + bd_{1}d_{2}\right)^{2} = a^{2} + 2abd_{1}d_{2} + b^{2}d^{2}_{1}d^{2}_{2}\,.
\end{equation*}
Subtracting, we obtain
\begin{equation*}
ab\left(d^{2}_{1} + d^{2}_{2} - 2d_{1}d_{2}\right) = 0\,,
\end{equation*}
i.e.
\begin{equation*}
ab\left(d_{1} - d_{2}\right)^{2} = 0\,.
\end{equation*}
By assumption, $d_{1} \neq d_{2}$, and since $b \neq 0$ in this case (the equation $f(\lambda)^{2} = a$ has at most one solution), we conclude $a = 0$. In particular, we obtain that
\begin{equation*}
\frac{f(d_{1})^{2}}{d^{2}_{1}} = \frac{f(d_{2})^{2}}{d^{2}_{2}} = b\,,
\end{equation*}
i.e.
\begin{equation}
\frac{f(d_{1})}{d_{1}} = \frac{f(d_{2})}{d_{2}}\,.
\label{LastEquationCommutativity}
\end{equation}
As explained in Remark \ref{RemarkOperatorMonotone}, the function $f$ being operator monotone, it implies that the function $x \to \frac{f(x)}{x}$ is operator monotone decreasing. Using that $h$ is strictly concave, it follows that $f$ is not of the form $f(x) = cx$, i.e. $x \to \frac{f(x)}{x}$ is injective. Therefore it follows from Equation \eqref{LastEquationCommutativity} that $d_{1} = d_{2}$, which is impossible. Then 
\begin{equation*}
a'_{i,j} = a'_{j,i} = 0\,, \qquad \left(1 \leq i \leq n_{1}\,, n_{1}+1 \leq j \leq n\right)\,,
\end{equation*}
and we get that the matrix $\A'$ is of the form
\begin{equation*}
\A' = \begin{pmatrix} \A'_{1} & 0 \\ 0 & \A'_{2} \end{pmatrix}\,, \qquad \left(\A'_{1} \in \Mat_{n_{1}}(\mathbb{D})\,, \A'_{2} \in \Mat_{n_{2}}(\mathbb{D})\right)\,.
\end{equation*}
Finally, $\D$ and $\A'$ commute, so $\Y$ and $\A$ commute, and using that
\begin{equation*}
\A\B = \A\Y\A\Y \qquad \text{ and } \qquad \B\A = \Y\A\Y\A\,,
\end{equation*}
we get that $\A\B = \B\A$\,.
\end{enumerate}
The theorem follows\,.

\end{proof}

\begin{remark}

The proof of Theorem \ref{TheoremSquareRootH} shows that the operator monotonicity assumption can be weakened. Indeed, the spectral part only uses that $f^{2}$ is increasing and strictly concave, while the commutativity part uses that the function
\begin{equation*}
x\mapsto \frac{f(x)}{x}
\end{equation*}
is injective on $\left(0\,, \infty\right)$. Therefore, the conclusion of Theorem \ref{TheoremSquareRootH} remains valid for any continuous function $f: \left(0\,, \infty\right) \to \left(0\,, \infty\right)$, with $f(1)=1$, and satisfying these two conditions\,.

\label{WeakerConditionsOnF}

\end{remark}

\begin{remark}

The results of this paper shows that the conjecture of Choi, Kim, and Lim (see \cite{CHOIKIMLIM}) can be extended to a large family of alternative means. However, the assumptions on the representing function cannot be completely removed. Indeed, if $f(t)=t$, then it follows from Equation \eqref{EquationRicatti}
\begin{equation*}
\A \hat{\sigma}_{f} \B = \left(\A^{-1} \sharp \B\right) \A \left(\A^{-1} \sharp \B\right) = \B
\end{equation*}
In particular, $\A \hat{\sigma}_{f} \B$ is always linearizable, independently of the matrices $\A$ and $\B$\,.

\noindent This naturally leads to the following problem: characterize the functions 
\begin{equation*}
f: \left(0\,, \infty\right) \to \left(0\,, \infty\right)
\end{equation*}
for which
\begin{equation*}
\A \hat{\sigma}_{f} \B \text{ is linearizable} \quad \Leftrightarrow \quad  \A\B=\B\A \text{ and } \left|\Spec(\A^{-1}\sharp\B)\right|\leq 2\,.
\end{equation*}
We note that the Wasserstein mean provides an interesting example in this direction. Indeed, its representing function is
\begin{equation*}
f(t)=\frac{1+t}{2}\,,
\end{equation*}
which is not of the form $f(t) = \sqrt{h(t)}$ with $h$ a non-affine operator monotone function, and also don't satisfy the conditions of Remark \ref{WeakerConditionsOnF}. Nevertheless, Theorem \ref{TheoremConjectureCKL} shows that the same rigidity phenomenon still holds\,.

\end{remark}

\begin{remark}

Theorem \ref{TheoremSquareRootH} applies to a large family of alternative means: 
\begin{enumerate}
\item The harmonic mean:
\begin{equation*}
h(x)=\frac{2x}{1+x}, \qquad f(x)=\sqrt{\frac{2x}{1+x}}\,.
\end{equation*}
\item The logarithmic mean:
\begin{equation*}
h(x)=\frac{x-1}{\log(x)}\,, \qquad f(x)=\sqrt{\frac{x-1}{\log(x)}}\,.
\end{equation*}
\item The weighted geometric means:
\begin{equation*}
h(x)=x^{\alpha}\,, \qquad f(x)=x^{\frac{\alpha}{2}}\,, \qquad \left(0 < \alpha < 1\right)\,.
\end{equation*}
\end{enumerate}
Therefore, for each of the above alternative means, linearizability implies
\begin{equation*}
\A\B=\B\A \qquad \text{and} \qquad \left|\Spec(\A^{-1}\sharp\B)\right|\leq 2\,.
\end{equation*}

\end{remark}

\end{document}